\documentclass[12pt]{amsart}
\usepackage{amscd,amssymb,latexsym}
\usepackage{fullpage}
\newcommand{\Mdef}[2]{\newcommand{#1}{\relax \ifmmode #2 \else $#2$\fi}}

\newcommand{\sm }{\wedge}

\newcommand{\tensor}{\otimes}

\newcommand{\sdr}{\rtimes}

\newcommand{\Hom}{\mathrm{Hom}}

\Mdef{\bhom}{\mathbf{\hat{H}om}}

\Mdef{\Mod}{\mathrm{mod}}

\newcommand{\st}{\; | \;}
\newcommand{\hash}{\#}

\newtheorem{thm}{Theorem}[section]
\newtheorem{lemma}[thm]{Lemma}
\newtheorem{prop}[thm]{Proposition}
\newtheorem{cor}[thm]{Corollary}

\theoremstyle{definition}

\newtheorem{defn}[thm]{Definition}

\newtheorem{notation}[thm]{Notation}

\newtheorem{example}[thm]{Example}

\newcommand{\qqed}{\qed \\[1ex]}
\renewenvironment{proof}[1][\hspace*{-.8ex}]{\noindent {\bf Proof #1:\;}}{\qqed}

\Mdef{\PH} {\Phi^H}
\Mdef{\PK} {\Phi^K}
\Mdef{\PL} {\Phi^L}
\Mdef{\PT} {\Phi^{\T}}

\Mdef{\ef}{E{\cF}_+}
\Mdef{\etf}{\widetilde{E}{\cF}}
\Mdef{\eg}{E{G}_+}
\Mdef{\etg}{\tilde{E}{G}}

\Mdef{\infl}{\mathrm{inf}}
\Mdef{\defl}{\mathrm{def}}
\Mdef{\res}{\mathrm{res}}
\Mdef{\ind}{\mathrm{ind}}
\Mdef{\coind}{\mathrm{coind}}

\Mdef{\univ}{\mathcal{U}}

\Mdef{\Fp}{\mathbb{F}_p}
\Mdef{\Zpinfty}{\Z /p^{\infty}}
\Mdef{\Zpadic}{\Z_p^{\wedge}}

\newcommand{\dichotomy}[2]{\left\{ \begin{array}{ll}#1\\#2 \end{array}\right.}

\newcommand{\adjunction}[4]{
\diagram
#1:#2 \rrto<0.7ex> &&
#3  \llto<0.7ex> :#4 
\enddiagram}
\newcommand{\lra}{\longrightarrow}
\newcommand{\lla}{\longleftarrow}

\newcommand{\lr}[1]{\langle #1 \rangle}

\newcommand{\Gspectra}{\mbox{$G$-{\bf spectra}}}

\Mdef{\we}{\mathbf{we}}
\Mdef{\fib}{\mathbf{fib}}
\Mdef{\cof}{\mathbf{cof}}
\Mdef{\BI}{\mathcal{BI}}

\newcommand{\colim}{\mathop{  \mathop{\mathrm {lim}} \limits_\rightarrow} \nolimits}

\Mdef{\B}{\mathbb{B}}
\Mdef{\C}{\mathbb{C}}
\Mdef{\D}{\mathbb{D}}
\Mdef{\E}{\mathbb{E}}
\Mdef{\T}{\mathbb{T}}
\Mdef{\F}{\mathbb{F}}
\Mdef{\G}{\mathbb{G}}
\Mdef{\I}{\mathbb{I}}
\Mdef{\N}{\mathbb{N}}
\Mdef{\Q}{\mathbb{Q}}
\Mdef{\R}{\mathbb{R}}
\Mdef{\bbS}{\mathbb{S}}
\Mdef{\Z}{\mathbb{Z}}

\Mdef{\bA}{\mathbb{A}}
\Mdef{\bB}{\mathbb{B}}
\Mdef{\bC}{\mathbb{C}}
\Mdef{\bD}{\mathbb{D}}
\Mdef{\bE}{\mathbb{E}}
\Mdef{\bF}{\mathbb{F}}
\Mdef{\bG}{\mathbb{G}}
\Mdef{\bH}{\mathbb{H}}
\Mdef{\bI}{\mathbb{I}}
\Mdef{\bJ}{\mathbb{J}}
\Mdef{\bK}{\mathbb{K}}
\Mdef{\bL}{\mathbb{L}}
\Mdef{\bM}{\mathbb{M}}
\Mdef{\bN}{\mathbb{N}}
\Mdef{\bO}{\mathbb{O}}
\Mdef{\bP}{\mathbb{P}}
\Mdef{\bQ}{\mathbb{Q}}
\Mdef{\bR}{\mathbb{R}}
\Mdef{\bS}{\mathbb{S}}
\Mdef{\bT}{\mathbb{T}}
\Mdef{\bU}{\mathbb{U}}
\Mdef{\bV}{\mathbb{V}}
\Mdef{\bW}{\mathbb{W}}
\Mdef{\bX}{\mathbb{X}}
\Mdef{\bY}{\mathbb{Y}}
\Mdef{\bZ}{\mathbb{Z}}

\Mdef{\cA}{\mathcal{A}}
\Mdef{\cB}{\mathcal{B}}
\Mdef{\cC}{\mathcal{C}}
\Mdef{\mcD}{\mathcal{D}} 
\Mdef{\cE}{\mathcal{E}}
\Mdef{\cF}{\mathcal{F}}
\Mdef{\cG}{\mathcal{G}}
\Mdef{\mcH}{\mathcal{H}} 
\Mdef{\cI}{\mathcal{I}}
\Mdef{\cJ}{\mathcal{J}}
\Mdef{\cK}{\mathcal{K}}
\Mdef{\mcL}{\mathcal{L}}

\Mdef{\cM}{\mathcal{M}}
\Mdef{\cN}{\mathcal{N}}
\Mdef{\cO}{\mathcal{O}}
\Mdef{\cP}{\mathcal{P}}
\Mdef{\cQ}{\mathcal{Q}}
\Mdef{\mcR}{\mathcal{R}}
\Mdef{\cS}{\mathcal{S}}
\Mdef{\cT}{\mathcal{T}}
\Mdef{\cU}{\mathcal{U}}
\Mdef{\cV}{\mathcal{V}}
\Mdef{\cW}{\mathcal{W}}
\Mdef{\cX}{\mathcal{X}}
\Mdef{\cY}{\mathcal{Y}}
\Mdef{\cZ}{\mathcal{Z}}

\Mdef{\ca}{\mathcal{a}}
\Mdef{\ct}{\mathcal{t}}

\Mdef{\At}{\tilde{A}}
\Mdef{\Bt}{\tilde{B}}
\Mdef{\Ct}{\tilde{C}}
\Mdef{\Et}{\tilde{E}}
\Mdef{\Ht}{\tilde{H}}
\Mdef{\Kt}{\tilde{K}}
\Mdef{\Lt}{\tilde{L}}
\Mdef{\Mt}{\tilde{M}}
\Mdef{\Nt}{\tilde{N}}
\Mdef{\Pt}{\tilde{P}}

\Mdef{\tA}{\tilde{A}}
\Mdef{\tB}{\tilde{B}}
\Mdef{\tC}{\tilde{C}}
\Mdef{\tE}{\tilde{E}}
\Mdef{\tH}{\tilde{H}}
\Mdef{\tK}{\tilde{K}}
\Mdef{\tL}{\tilde{L}}
\Mdef{\tM}{\tilde{M}}
\Mdef{\tN}{\tilde{N}}
\Mdef{\tP}{\tilde{P}}

\Mdef{\ft}{\tilde{f}}
\Mdef{\xt}{\tilde{x}}
\Mdef{\yt}{\tilde{y}}

\Mdef{\Ab}{\overline{A}}
\Mdef{\Bb}{\overline{B}}
\Mdef{\Cb}{\overline{C}}
\Mdef{\Db}{\overline{D}}
\Mdef{\Eb}{\overline{E}}
\Mdef{\Fb}{\overline{F}}
\Mdef{\Gb}{\overline{G}}
\Mdef{\Hb}{\overline{H}}
\Mdef{\Ib}{\overline{I}}
\Mdef{\Jb}{\overline{J}}
\Mdef{\Kb}{\overline{K}}
\Mdef{\Lb}{\overline{L}}
\Mdef{\Mb}{\overline{M}}
\Mdef{\Nb}{\overline{N}}
\Mdef{\Ob}{\overline{O}}
\Mdef{\Pb}{\overline{P}}
\Mdef{\Qb}{\overline{Q}}
\Mdef{\Rb}{\overline{R}}
\Mdef{\Sb}{\overline{S}}
\Mdef{\Tb}{\overline{T}}
\Mdef{\Ub}{\overline{U}}
\Mdef{\Vb}{\overline{V}}
\Mdef{\Wb}{\overline{W}}
\Mdef{\Xb}{\overline{X}}
\Mdef{\Yb}{\overline{Y}}
\Mdef{\Zb}{\overline{Z}}

\Mdef{\db}{\overline{d}}
\Mdef{\hb}{\overline{h}}
\Mdef{\qb}{\overline{q}}
\Mdef{\rb}{\overline{r}}
\Mdef{\tb}{\overline{t}}
\Mdef{\ub}{\overline{u}}
\Mdef{\vb}{\overline{v}}

\Mdef{\hc}{\hat{c}}
\Mdef{\he}{\hat{e}}
\Mdef{\hf}{\hat{f}}
\Mdef{\hA}{\hat{A}}
\Mdef{\hH}{\hat{H}}
\Mdef{\hJ}{\hat{J}}
\Mdef{\hM}{\hat{M}}
\Mdef{\hP}{\hat{P}}
\Mdef{\hQ}{\hat{Q}}

\Mdef{\thetab}{\overline{\theta}}
\Mdef{\phib}{\overline{\phi}}

\Mdef{\uA}{\underline{A}}
\Mdef{\uB}{\underline{B}}
\Mdef{\uC}{\underline{C}}
\Mdef{\uD}{\underline{D}}

\Mdef{\bolda}{\mathbf{a}}
\Mdef{\boldb}{\mathbf{b}}
\Mdef{\bfD}{\mathbf{D}}

\Mdef{\fm}{\frak{m}}
\Mdef{\fp}{\frak{p}}

\newcommand{\fX}{\mathfrak{X}}

\Mdef{\eps}{\epsilon}

\newcommand{\up}{\mathsf{V}}

\newcommand{\cell}{\mathrm{Cell}}

\input{xypic}
\usepackage{graphicx}
\newcommand{\sub}{\mathrm{Sub}}
\newcommand{\PP}{\mathbb{P}}
\newcommand{\cEi}{\cE^{-1}}
\newcommand{\cIi}{\cI^{-1}}
\newcommand{\Hbar}{\overline{H}}
\newcommand{\bbT}{\mathbb{T}}
\newcommand{\Zt}{\tilde{\Z}}
\newcommand{\Qt}{\tilde{\Q}}
\newcommand{\gb}{\overline{g}}
\newcommand{\cotoral}{\leq_{ct}}
\newcommand{\fa}{\mathfrak{a}}
\newcommand{\Wbar}{\overline{W}}
\newcommand{\End}{\mathrm{End}}
\renewcommand{\tb}{\overline{\times}}
\newcommand{\nt}{\mathbf{nt}}
\newcommand{\full}{\mathrm{full}}

\newcommand{\cNbar}{\overline{\cN}}
\newcommand{\Ah}{\hat{A}}
\newcommand{\wt}{\tilde{w}}
\newcommand{\Th}{\mathrm{Th}}
\newcommand{\diag}{\mathrm{diag}}
\newcommand{\spl}{\mathrm{Split}}
\newcommand{\gammaperp}{\gamma^{\perp}}
\newcommand{\width}{\mathrm{width}}
\newcommand{\cOcF}{\mathcal{O}_{\cF}}
\newcommand{\cOcFt}{\widetilde{\mathcal{O}}_{\cF}}
\newcommand{\Vt}{\tilde{V}}
\newcommand{\e}{(e)}
\newcommand{\Qmod}{\Q\mbox{-mod}}
\newcommand{\cVGfull}{\cV^G_{\full}}
\newcommand{\Rt}{\tilde{R}}
    \newcommand{\cospan}{\lrcorner}
    \newcommand{\module}{\mbox{-mod}}
\newcommand{\cFa}{\cF_a}
\newcommand{\cFb}{\cF_b}
\newcommand{\cOcFa}{\cO_{\cFa}}
\newcommand{\cOcFb}{\cO_{\cFb}}

\newcommand{\cOcK}{\cO_{\cK}}

\newcommand{\siftyV}[1]{S^{\infty V( #1)}}

 \newcommand{\efp}{E\cF_+}
      \newcommand{\efap}{E\cF^a_+}
      \newcommand{\efbp}{E\cF^b_+}
      \newcommand{\cEhi}{\hat{\cE}^{-1}}
      \newcommand{\Sh}{\hat{S}}

\newcommand{\freeGspectra}{\mbox{free-$G$-spectra}}
\newcommand{\cofreeGspectra}{\mbox{cofree-$G$-spectra}}
\newcommand{\HBGmodules}{\mbox{$H^*(BG)$-mod}}
\newcommand{\torsionHBGmodules}{\mbox{tors-$H^*(BG)$-mod}}
\newcommand{\completeHBGmodules}{\mbox{comp-$H^*(BG)$-mod}}
\newcommand{\moduleGspectra}{\mbox{-mod-$G$-spectra}}

\newcommand{\modulespectra}{\mbox{-mod-spectra}}
\newcommand{\modules}{\mbox{-mod}}

\newcommand{\Tt}{\tilde{T}}
\newcommand{\cOcFZ}{\cO_{\cF/Z}}
\newcommand{\cOcFTt}{\cO_{\cF/\Tt}}

\newcommand{\utwo}{{U(2)}}

\begin{document}
\title{An algebraic model for rational $U(2)$-spectra}

\author{J.P.C.Greenlees}
\address{Mathematics Institute, Zeeman Building, Coventry CV4, 7AL, UK}
\email{john.greenlees@warwick.ac.uk}

\date{}

\begin{abstract}
We construct an explicit and calculable models for rational 
$U(2)$-spectra. This is obtained by assembling seven blocks obtained 
in previous work: the toral part from \cite{AGtoral, gtoralq} and the 
work on small toral groups \cite{t2wqalg,t2wqmixed,gq1}. The assembly 
process requires detailed input on fusion and Weyl groups. 
\end{abstract}

\thanks{The author is grateful for comments, discussions  and related
  collaborations with S.Balchin, D.Barnes, T.Barthel, M.Kedziorek,
  L.Pol, J.Williamson. The work is partially supported by EPSRC Grant
  EP/W036320/1. The author  would also  like to thank the Isaac Newton
  Institute for Mathematical Sciences, Cambridge, for support and
  hospitality during the programme Equivariant Homotopy Theory in
  Context, where later parts of  work on this paper was undertaken. This work was supported by EPSRC grant EP/Z000580/1.  } 

\maketitle

\tableofcontents
\section{Overview}
For compact Lie groups $G$ there is growing evidence for the
conjecture \cite{AGconj} that there is an abelian category $\cA (G)$
and a Quillen equivalence
$$\Gspectra \simeq DG-\cA (G). $$
The general philosophy is that for any group $G$ one may break up
the category using idempotents of the Burnside ring into pieces. Each
piece is dominated by a subgroup $H$. One hopes that the  component of
$G$-spectra  dominated by  $H$ can be largely reduced to consideration of
$H$-spectra  dominated by $H$. Finally, the component of $H$-spectra
dominated by $H$ has the general shape of a ``tensor product'' of models
for toral groups.

The Quillen equivalence when $G$ is a torus was proved in
\cite{tnqcore} and the case when $G$ has identity component a torus
is tackled in \cite{t2wqalg, t2wqmixed, gq1}, and proved for many cases of
rank $\leq 2$.

The purpose of this paper is to show how the previous work on toral
groups can be
assembled to give a complete, explicit and calculable model when
$G=U(2)$. This is instructive because it shows that the assembly 
process requires detailed work on fusion and Weyl groups. For this
small group we can deal with the issues in an ad hoc way, and it
highlights places where systematic results would be illuminating.

\subsection{The results}
The abelian category $\cA (G)$ takes the form of a category of sheaves
of modules over the space $\fX_G=\sub(G)/G$ of conjugacy classes of
(closed) subgroups of $G$.

It is known from \cite{so3q, KedziorekSO(3)} that the category of
rational $SO(3)$-spectra splits as a product of 7 blocks. Of these, 5
are 0-dimensional (dominated by $SO(3), A_5, \Sigma_4, A_4, D_4$)
and 2 are 1-dimensional (dominated by the maximal torus ($SO(2)$) and its
normalizer ($O(2)$)).  The word `dimension' applies both to the space
of subgroups and to the injective dimension of the abelian model. 

The overall conclusion of the present paper is that using the quotient map $U(2)\lra
PU(2)\cong SO(3)$, the category of rational $U(2)$-spectra
breaks up as the corresponding product of 7 blocks. The blocks of
$U(2)$-spectra are one dimension larger than those in $SO(3)$, so that 5 of them are
1-dimensional and 2 of them are 2-dimensional. Each of the 7 blocks is
dominated by the inverse image of the corresponding subgroup in
$SO(3)$.

This conclusion is very clean, but the precise structure of the
algebraic models is affected by the details of the group theory of $U(2)$.

\subsection{Associated work in preparation}
This paper is the fourth in a series  of 5 constructing an algebraic
category $\cA (SU(3))$ and showing it gives an algebraic model for 
rational $SU(3)$-spectra. This series gives a concrete 
illustrations of general results in small and accessible 
examples.

The first paper \cite{t2wqalg} describes the group theoretic data that feeds into the construction of an
abelian category $\cA (G)$ for a toral group $G$ and makes it
explicit for toral subgroups of rank 2 connected groups.

The second paper \cite{gq1} constructs algebraic models for all relevant 1-dimensional 
blocks, and the results are applied in the present
paper to each of the five 1-dimensional blocks.
The third paper \cite{t2wqmixed} constructs algebraic models for
blocks of rank 2 toral groups of mixed type, and the results are
applied in the present paper to cover the block dominated by the
normalizer of the maximal torus.

Finally, the paper \cite{su3q} constructs $\cA (SU(3))$
in 18 blocks and shows it is equivalent to the category of rational
$SU(3)$-spectra. The most complicated parts of the model for
$SU(3)$ are the blocks from $U(2)$, as described in the present paper.

This series is part of a more general programme. Future installments
will consider blocks with Noetherian Balmer spectra \cite{AGnoeth} and
those with no cotoral inclusions \cite{gqwf}. 
An account of the general nature of the models is in preparation
\cite{AVmodel}, and the author hopes that this will be the basis of the proof that the
category of rational $G$-spectra has an algebraic model in general.

\subsection{Contents}
In Section \ref{sec:sub} we explain how the space of subgroups can be 
broken into blocks, giving a decomposition of the category of 
$G$-spectra into a product of pieces each dominated by a particular 
subgroup. This is illustrated for proper subgroups of $U(2)$.

In Section \ref{sec:dimonemodels} we describe the models for 
1-dimensional blocks as given in \cite{gq1}. In Section 
\ref{sec:partu2} we describe the partition of subgroups of $U(2)$ into 
7 blocks, of which 5 are 1-dimensional. In Section \ref{sec:fiveeasy}
we describe the model for each of the five 1-dimensional blocks. 
To describe the 2-dimensional blocks we need a systematic nomenclature 
for subgroups, which is introduced in Section \ref{sec:subT}. The 
2-dimensional blocks of $G$ are the ones dominated by (i) 
the maximal torus and (ii) its normalizer. In 
earlier work, we have described the models for the corresponding 
blocks of subgroups of $H$, and in Section \ref{sec:normfus} we 
describe the fusion when passing from $H$-conjugacy to $G$-conjugacy. 
Finally, in Section  \ref{sec:domT} we give the model for the block 
dominated by the maximal torus and in Section \ref{sec:domN} we give 
the model for the normalizer of the maximal torus. 


\subsection{Notation}
The ambient group throughout this paper is $G=U(2)$, although we will
usually write the full name. We write $\T$ for the maximal torus of
diagonal matrices, $\N=N_{U(2)}(\T)$ for its normalizer,
$Z$ for the centre of scalar matrices, $\Tt$ for
the diagonal maximal torus of $SU(2)$ (consisting of matices $\diag
(\lambda, \lambda^{-1})$). We write $W$ for the Weyl group (of order
2), and $\Lambda_0=H_1(\T)=\Z W$ for the toral lattice.

Since $Z\cap \Tt$ is of order 2, we often need to consider central
products, and if $A\subseteq Z, B\subseteq SU(2)$ we write $A\times_2 B$
for the image of $A\times B$ in $U(2)=(Z\times SU(2))/C_2$ under the
central quotient. 

\section{Spaces of subgroups}
\label{sec:sub}

\subsection{Isotropic localizations}
For any collection $\cV$ of subgroups closed under conjugation we
write $\Gspectra \lr{\cV}$ for the category of rational $G$-spectra
with geometric isotropy in $\cV$. For various groups $G$ and
collections $\cV$ we will show that there is an abelian category $\cA
(G|\cV)$ and a Quillen equivalence
$$\Gspectra\lr{\cV}\simeq DG-\cA (G|\cV), $$
and in this case we say $\cA (G|\cV)$ is an abelian model for
$G$-spectra over $\cV$.

The purpose of this note is to describe $\cA (G)$ and to prove the
conjecture when  $G=U(2)$ and $\cV$ consists of all subgroups.

\subsection{Spaces of subgroups and the Burnside ring}
Choosing a bi-invariant metric on $G$, we may consider the space $\sub
(G)$ as a metric space with the Hausdorff metric. We refer to the
quotient topology on the space
$\fX_G=\sub(G)/G$ of conjugacy classes as the {\em h-topology}. In
addition, the space $\fX_G$ also has the coarser Zariski topology, whose
closed sets are h-closed sets that are also closed under passage to
cotoral\footnote{$K$ is {\em cotoral}  in $H$ if $K$ is normal in $H$
  with $H/K$ a torus} subgroups.

We write $\cF
G\subseteq \sub (G)$ for the subspace of subgroups with finite Weyl
group and $\Phi G\subseteq \sub (G)/G$ for the corresponding space of
conjugacy classes.
 T.tom Dieck showed that  the rational mark homomorphism taking degrees
of geometric fixed points induces an isomorphism
$$[S^0,S^0]^G\stackrel{\cong}\lra C(\Phi G, \Q).$$
 For a map $f:S^0\lra S^0$, the
Borel-Hsiang-Quillen Localization Theorem says that if $K$ is
cotoral in $H$ then the degree of $\Phi^Kf$ is the same as that for
$\Phi^Hf$. Thus if $A$ is an h-clopen set of conjugacy classes closed
under cotoral specialization there is an idempotent $e_A: S^0\lra S^0$
with support $A$.

\subsection{Partitions of spaces of subgroups}
The existence of idempotents shows that  a partition 
$$\fX_G =\cV_1\amalg \cdots \amalg \cV_N$$
where $\cV_i$ is clopen in the h-topology and closed under
cotoral specilization gives rise to an equivalence
$$\Gspectra\simeq \Gspectra\lr{\cV_1}\times \cdots \times
\Gspectra\lr{\cV_N}. $$

Similarly, the form of $\cA (G|\cV)$ is that of a category of sheaves over $\cV$,
so that the decomposition of $\fX_G$ above gives
$$\cA (G)=\cA (G|\fX_G)=\cA (G|\cV_1)\times \cdots \times \cA (G|\cV_N). $$
The decompositions of the category of $G$-spectra correspond to those
of the algebraic model, so the model for $G$-spectra can be
established by showing $\Gspectra\lr{\cV_i}\simeq \cA (G|\cV_i)$ for
each $i$.

\subsection{Dominated components}
\label{subsec:dom}
One may prove \cite{AVmodel} that for any group $G$ there is a
partition as above where each of the terms $\cV_i$ has a dominant
subgroup. We say that $\cV$ has  {\em dominant 
  subgroup} $H$  if $H$ has finite Weyl group and 
controls $\cV$ in the sense that there is a chosen neighbourhood $\cN_H$ of $H$ in $\Phi H$ whose 
image $\cNbar_H$  in $\sub (G)/G$ lies in  $\Phi G$, and 
$$\cV^G_H:=\{ (K)_G \st K\leq_{ct} H' \mbox{ with } (H')\in \cN_H\}$$
is the cotoral down-closure of $\cN_H$ up to $G$-conjugacy. 

We note that the conjugacy class of $H$ is maximal in $\cV$, but $H$
does not determine $\cV$ since we also need to choose the neighbourhood $\cN_H$. 
Nonetheless, in each case we will make a choice and write $\cV^G_H$ for a chosen component dominated by $H$. 

It is not necessary to appeal to the general result since we will give
an explicit decomposition for $U(2) $.  
We may tabulate the dominant subgroups $H$ and their components as
follows. Any element normalizing  $H$ normalizes the identity
component, so $H$ is determined by $H_e$ and the finite subgroup $H_d=H/H_e$
of $W_G(H_e)$. Accordingly we will index $H$ by the pair $(H_e, F)$ where $F$ is a
finite subgroup of $W_G(H_e)$. In the rest of the section we will give
 more details of $\fX_G$ for proper subgroups $G$ of $U(2)$, and from Section
 \ref{sec:partu2} onwards we give more details on $G=U(2)$ itself.

 In the following table, $T$ is the maximal torus of $SO(3)$ (a
 circle),  $\T$  is the maximal torus of $U(2)$ (diagonal
 matrices, a 2-torus) and $Z$ is the centre of $U(2)$ (the scalar matrices, a
 circle). The entry `Discrete' means that the geometric isotropy
 space is discrete.  
\label{sec:proper}
$$\begin{array}{ll|ccc|}
G&(H_e, F)&\dim (\cV^G_{(H_e,F)})&W_G(H)& \\
\hline 
SO(2)&(SO(2),1)&1&1&[8]\\
\hline 
O(2)&(SO(2),1)&1&C_2&[8]\\
&(SO(2),C_2)&1&1&[8]\\
\hline 
SO(3)&(T, 1) &1&C_2&[8]\\
&(T, C_2) &1&1&[8]\\
&(SO(3),1)&0&1&\mathrm{Discrete}\\
&(1, A_5)&0&1&\mathrm{Discrete}\\
&(1, \Sigma_4)&0&1&\mathrm{Discrete}\\
&(1, A_4)&0&C_2&\mathrm{Discrete}\\
&(1, D_4)&0&\Sigma_3&\mathrm{Discrete}\\
\hline 
U(2)&(\T, 1) &2&C_2&[8]\\
&(\T, C_2) &2&1&[8]\\
&(U(2),1)&1&1&\mathrm{Discrete}\\
&(Z, A_5)&1&1&\mathrm{Discrete}\\
&(Z, \Sigma_4)&1&1&\mathrm{Discrete}\\
&(Z, A_4)&1&C_2&\mathrm{Discrete}\\
&(Z, D_4)&1&\Sigma_3&\mathrm{Discrete}\\
\hline 
\end{array}$$

\subsection{The group $SO(2)$}
The closed subgroups of $SO(2)$ are the finite cyclic subgroups and 
$SO(2)$ itself, so that with the h-topology we have 
$$\fX_{SO(2)}=\cC^\hash, $$
the one point  compactification 
of the discrete space $\cC=\{ C_n \st 
n\geq 1\}$ of cyclic subgroups. We note that $\cC^\hash $ has dominant
subgoup $SO(2)$, so
$$\fX_{SO(2)}=\cV^{SO(2)}_{SO(2)}.$$

\subsection{The group $O(2)$}
The closed subroups of $O(2)$ are the closed subgroups of $SO(2)$
together with a single conjugacy class of  dihedral  subgroups of
order $2n$ for each $n\geq 1$ and 
$O(2)$ itself. Thus, with the h-topology we have 
$$\fX_{O(2)}=\cC^\hash \amalg \mcD^\hash , $$
where $\cC$ is as before and 
$$\mcD=\{ (D_{2n}) \st  n\geq 1\}. $$
We note that $\cC^\hash$ has dominant subgroup $SO(2)$ and
$\mcD^\hash$ has dominant subgroup $O(2)$ so we
have
$$\fX_{O(2)}=\cV^{O(2)}_{SO(2)}\amalg \cV^{O(2)}_{O(2)}. $$

\subsection{The partition for the group $SO(3)$}
\label{subsec:so3}
The conjugacy classes of subgroups of $SO(3)$ admit a partition into seven pieces 
$$\fX_{SO(3)}=\cV^{SO(3)}_{SO(2)}\amalg \cV^{SO(3)}_{O(2)}\amalg 
\cV^{SO(3)}_{SO(3)}\amalg  \cV^{SO(3)}_{A_5}\amalg
\cV^{SO(3)}_{\Sigma_4}\amalg \cV^{SO(3)}_{A_4}  \amalg
\cV^{SO(3)}_{D_4}   $$
as follows. In the five cases  $H\in \{ SO(3), A_5, \Sigma_4, A_4,
D_4\}$ the space $\cV^{SO(3)}_H=\{ (H)\}$ is a singleton. If $H=SO(2)$
we have the space of cyclic subgroups
  $$\cV^{SO(3)}_{SO(2)}=\{ (C_{n})\st n\geq 1\}\cup \{ (SO(2))\} $$
 
Finally, when $H=O(2)$, we have the space of dihedral subgroups
$$\cV^{SO(3)}_{O(2)}=\{ (D_{2n})\st n\geq 3\}\cup \{ (O(2))\}. $$
Note that this starts with $D_6$: in $SO(3)$ the dihedral group $D_2$ is conjugate to the
cyclic subgroup $C_2$ of $SO(2)$, and $D_4$ has been treated as
exceptional because its Weyl group is bigger than the others. 
  
All seven sets are  clopen in the Hausdorff metric topology and closed
under cotoral specialization. This therefore gives a partition of the
Zariski space $\fX_{U(2)}$ into seven Zariski clopen pieces. 

\section{Classic models}
\label{sec:dimonemodels}

We describe here the known low dimensional models in a standard
form. These will all be used to model blocks of $G$-spectra when $G$
is $U(2)$ or one of its subgroups.  
\subsection{0-dimensional}
From \cite{gfreeq2} we see that $G$-spectra with singleton geometric
isotropy has a simple model:
$$\Gspectra\lr{(H)}\simeq DG-\cA(G|(H))$$
where 
$$\cA(G|(H))=\mbox{tors-}H^*(BW_G^e(H))[W_G^d(H)]\modules.$$
This gives a model for $G$-spectra with geometric isotropy in $\cV$
for any discrete space $\cV$. This describes the model in the cases
marked `discrete' in the table in Subsection \ref{subsec:dom}. 

\subsection{1-dimensional (generalities)}
We describe the construction of the model $\cA (G|\cV)$ in the
simplest cases, typified by $G=SO(2)$ or $O(2)$.  We suppose that
there is a countable set $\cK$ with $\cV\cong \cK^\hash$; since we are
thinking of $\cV^G_H$ we write $H$ for the compactifying point. The
set $\cK$ is countable, and we will say that something holds `for
almost all $k$' if it holds for all subgroups cotoral in $H$ and for
all but finitely many of the rest.

In general, to specify the model $\cA (\cK, \mcR, \cW)$ we need additional data:
\begin{itemize}
\item a system $\mcR$
  of rings (i.e.,  rings $\mcR (H)$, $\mcR (k)$ for $k\in \cK$. (In
  general we will need ring maps $\mcR (H)\lra \mcR (k)$ for almost
  all $k$, but in our case $\mcR(H)=\Q$, so this is not additional structure)).
\item  a component structure $\cW$ (i.e., finite groups $\cW_H$ and 
$\cW_k$ for $k\in \cK$ together with group homomorphisms $\cW_k \lra
\cW_H$ for almost all $k$.)
\item the group $\cW_H$ acts on $\mcR (H)$ and the group $\cW_k$ acts
  on $\mcR (k)$.
\end{itemize}
  
In our case $\mcR(k)$ and $\cW_k$ will be independent of $k$ and $\mcR
(H)=\Q$. We simplify the notation accordingly, writing $\cA (\cK, R,
\cW)$ where $\mcR(k)=R$ for all $k$. When $\cW$ is trivial in the sense that $\cW_k=\cW_H=1$, we
simply write $\cA (\cK, R)$ for the model, and if $\cW$ is constant at
$W$ we will see that $\cA (\cK, R, \cW)=\cA (\cK,R)[W]$

There are two subcases with slightly different characters. 
\subsection{1-dimensional (Type 1)}
For Type 1, we  take $R=\Q[c]$ where $c$ is of degree $-2$. We then take $\cOcK=\prod_{k\in \cK}
\Q[c]$, together with the  multiplicatively closed set 
$$\cE =\{ c^v\st v:\cK\lra \N \mbox{ which is } 0 \mbox{ almost 
  everywhere }\}. $$
The {\em standard model} 
$\cA(\cK, \Q[c], \cW)$ has objects $\beta: N\lra \cEi \cOcK\tensor V$ where $N$ is
a $\cOcK$-module and $V$ is a graded rational vector space, and the
map $\beta $ is inverting the multiplicatively closed set $\cE$.
The idempotent summand $e_k N$ has an 
action of $\cW_k$, the vector space $V$ has an action of $\cW_H$ 
and the composite $e_kN\lra N \lra \cEi \cOcK \tensor V\lra 
\Q[c,c^{-1}]\tensor V$ is $\cW_k$-equivariant.

  \begin{example}
    (i) The standard model for rational $SO(2)$-spectra is a Type 1 
    model: $\cA (SO(2))=\cA (\cC, \Q[c])$. 

    (ii) The standard model for rational $O(2)$ spectra with cyclic 
    geometric isotropy is a Type 1 
    model:
    $$\cA (O(2)|\cC)=\cA (\cC, \Q[c], C_2)=\cA (\cC, \Q[c])[C_2],$$ 
    where the component structure is constant at the group of order 2, 
    acting on the ring $\Q[c]$ to negate $c$. 
        \end{example}
  
\subsection{1-dimensional (Type 0)}
In Type 0, we take $\cA (\cK , \Q, W)$ to be sheaves $\cF$ of $\Q$-modules over
$\cK^\hash$. We then take $\cOcK=\prod_{k\in \cK}
\Q$, together with the  multiplicatively closed set 
$$\cI =\{ s: \cK\lra \{0,1\} \st s \mbox{ is } 1 \mbox{ almost 
  everywhere }\}. $$
The stalk $\cF_k$ has an 
action of $\cW_k$, the stalk $\cF_H $ has an action of $\cW_H$ 
and the horizontal spreading map $\cF_H\lra \cIi \prod_k \cF_k$ is 
 $\cW_k$-equivariant. 
  \begin{example}
     The standard model for rational $O(2)$-spectra with dihedral 
    geometric isotropy is a Type 1 
    model: $\cA (O(2)|\cV^{O(2)}_{O(2)})=\cA (\mcD, \Q, C'_2)$, 
    where the component structure is $C_2$ at all points of $\mcD$ and
    the trivial group at $O(2)$.
  \end{example}

\section{The partition for the group $U(2)$}
\label{sec:partu2}
The centre $Z$ of $U(2)$ consists  of scalar matrices (a circle group). 
In particular this means the subgroups  in $G=U(2)$ with finite Weyl group all contain the 
centre so that the model is closely related to that for
$PSU(2)=U(2)/Z\cong SO(3)$.

Using the decompositions for subgroups of $SO(3)$ described in
Subsection \ref{subsec:so3}, we 
obtain a corresponding decomposition into 7 pieces. Indeed, factoring
out the centre gives a map 
$$p: U(2)\lra PU(2)=SO(3). $$
In fact $U(2)$ is the central product $SU(2)\times_2 Z$ (where
$\times_2$ denotes the quotient by the central subgroup $C_2$), and the
composite $SU(2)\lra U(2)\lra PU(2)=SO(3)$ is the quotient of $SU(2)$
by its central subgroup of order 2. For any subgroup $K\subseteq SO(3)$  we write $\Kt$ for
its double cover $\Kt=p^{-1}K\cap SU(2)$ and
$p^{-1}K=\Kt\times_{C_2}Z$.

\begin{lemma}
 The function $p$ induces a Zariski continuous function 
$$p_*: \fX_{U(2)}\lra \fX_{SO(3)}. $$
\end{lemma}

\begin{proof}
If one quotes the fact \cite{spcgq} that  $\fX_G$ is the Balmer
spectrum with the given topology \cite{prismatic}, this follows from
the fact that inflation $\mbox{$SO(3)$-spectra}  \lra
\mbox{$U(2)$-spectra}$ is monoidal and therefore induces a map on
Balmer spectra.

Alternatively, we may check this directly. 
It is obvious that $p_*$ is h-continuous and if $V\subseteq
\fX_{SO(3)}$ is closed under cotoral specialization so is
$p_*^{-1}(V)$. Indeed, if $p(a)=x\in V$ and $b$ is cotoral in $a$ then
$p(b)$ is cotoral in $p(a)=x$ ($p(a)=a/a\cap Z$, $p(b)=b/b\cap Z$ and 
$p(a)/p(b)=(a/a\cap Z)/(b/b\cap Z)=a/(b\cdot a\cap Z)$ is a quotient
of the torus $a/b$). Thus $b \in p_*^{-1}(V)$ as required.
\end{proof}

This means the partition of $SO(3)$ into the seven pieces 
gives a clopen partition of $\fX_{U(2)}$ into seven pieces.

\begin{lemma}
If $H$ is one of the 7 named subgroups of $SO(3)$ occurring in the partition
$$p_*^{-1}\cV^{SO(3)}_H=\cV^{U(2)}_{p^{-1}H}.  $$
In other words, the group $p^{-1}H=\Ht\times_{C_2}Z $
is dominant in $p_*^{-1}\cV^{SO(3)}_H$ (it is maximal and the set is the
closure under cotoral specialization of a neighbourhood of $p^{-1}H$
in the space of subgroups with finite Weyl group).
\end{lemma}

\begin{proof}
Note that if $K$ lies in $p_*^{-1}(\cV^{SO(3)}_H)$ then $p(K)\subseteq (H)$ and so
$\Kt=p^{-1}(K) \subseteq \Ht$ and $K\subseteq \Ht\times_{C_2}T$.
\end{proof}

This gives a partition 
$$\fX_{U(2)}=\cV^{U(2)}_{\T}\amalg \cV^{U(2)}_{\N}\amalg 
\cV^{U(2)}_{U(2)}\amalg 
\cV^{U(2)}_{\tilde{A}_5\times_{C_2}Z}\amalg
\cV^{U(2)}_{\tilde{\Sigma}_4\times_{C_2}Z}\amalg
\cV^{U(2)}_{\tilde{A}_4 \times_{C_2}Z}\amalg
\cV^{U(2)}_{\tilde{D}_4\times_{C_2}Z} .  $$

This is the basis of the model. We need describe the topology and
additional structure for each of the components.
The pieces coming from isolated conjugacy classes in
$SO(3)$ are 1-dimensional and the model will then be of the form $\cA (\cK, R, \cW)$ described
in Section \ref{sec:dimonemodels}.  The remaining two are more complicated
and require detailed discussion. We will take these in order of increasing
difficulty. 

\section{Five easy pieces}
\label{sec:fiveeasy}
Throughout this section the group $H\in \{SO(3), A_5, \Sigma_4, A_4,
D_4\}$ is one of the totally isolated subgroups of $SO(3)$. We have noted
that the inverse image in $U(2)$ is of the form
$p^{-1}H=\Ht\times_{C_2}Z$, where $\Ht$ is the double cover of
$H$. We thus have an extension 
$$1\lra Z \lra p^{-1}H \stackrel{p}\lra H\lra 1. $$
If $H=SO(3)$ we will make a separate argument, but if $H$ is finite,
then  $p^{-1}H$ is a toral subgroup with component group $H$ so that we
can apply the methods of \cite{t2wqalg}. We recall that a subgroup $K$
of $p^{-1}(H)$ is said to be {\em full} if $p(K)=H$.

\subsection{Five easy spaces}
The space of subgroups of $U(2)$ mapping to $H$ isolated subgroups has  a simple form.  

\begin{lemma}
  \label{lem:fiveeasy}
(i) If  $K$ is a proper full subgroup of $p^{-1}H$ then $K\cap Z$ is
of even order.

(ii) {\em (Classification up to conjugacy in $p^{-1}(H)$)}
The set of $p^{-1}H$-conjugacy classes of full subgoups  has a single subgroup 
$\Ht\times_{C_2} Z$ of top dimension.  For each $s\geq 1$ there are 
$a(H)$ conjugacy classes meeting $Z$ in $C_{2s}$, where $a(H)=1, 1, 2, 3,
4$ for $H=SO(3), A_5,\Sigma_4, A_4, D_4$.  All are normal in
$p^{-1}H$. One of the conjugacy classes is represented by the
canonical 
lift $\Ht_s:=\langle \Ht, C_{2s}\rangle$.

(iii) {\em (Classification up to conjugacy in $U(2)$)}
For each $s$ the non-canonical subgroups become conjugate in
$U(2)$, so the set of $U(2)$-conjugacy classes of full subgoups  has a single subgroup 
$\Ht\times_{C_2} Z$ of top dimension and for each $s\geq 1$ there are 
$b(H)$-conjugacy classes,  $b(H)=1, 1, 2, 2, 2$ for $H=SO(3), A_5,\Sigma_4, A_4, D_4$.  
 \end{lemma}

    \begin{proof}
      (i) It is clear that $p^{-1}H=\Ht \times_{C_2}Z$ that  if $p(K)=H$ then 
      $K \subseteq \Ht \times_{C_2} T$.   Now suppose  $K\cap Z=C_n$
      is of odd order, since $H^2(H;C_n)=0$, the extension splits and
      $K\cong C_n\times H$. Representation theory then shows that
      $H\times C_n$ does not have a faithful 2-dimensional representation
      for $n\geq 2$.
      
For Parts (ii) and (iii), we suppose first that $H=SO(3)$, so that
$p^{-1}(H)=U(2)$. If $K$ is of 
rank 2 then $K=U(2)$. Otherwise $K$ is of rank 1, and since it maps on 
to $SO(3)$ it is not a circle. The identity component of $K$ cannot be 
$SO(3)$ since $SO(3)$ does not have a faithful 2 dimensional 
representation, so $K_e\cong SU(2)$, and there is a single conjugacy 
class of subgroups of this type so we may suppose $K_e=SU(2)$, and by 
Part (i) $K=SU(2)\times_{C_2}C_{2s}$. Part (iii) is the same as Part
(ii) when $H=SO(3)$. 
      
Now suppose  $H$ is finite. Thus
we suppose $K$ is a full subgroup of the toral group $p^{-1}(H)$. 

(ii) Since $H^2(H;Z)=H^3(H; \Z)$ is of order 2, there is a subgroup $K$ with
$C_n=K\cap Z$ if and only if $n$ is even. By \cite[Lemma 3.3]{t2wqalg}, the $p^{-1}H$-conjugacy
classes of such subgroups are in bijection to $H^1(H; Z/C_s)\cong \Hom
(H,T)=\Hom (H^{ab}, T)$, and hence in bijection to the abelianization
$H^{ab}$. This gives the numbers $a(H)$. 

(iii) Since $Z$ is the centre of $U(2)$, any further $U(2)$ conjugacy comes
through $PU(2)=SO(3)$. The normalizers of the subgroups $H$ are
well known and recorded in Lemma \ref{lem:weyleasy} below. They are
non-trivial only for $A_4$ and $D_4$, and in both cases they act
transitively on the non-canonical conjugacy classes. This gives the
numbers $b(H)$.
  \end{proof}

\begin{cor}
For $H\in \{SO(3), A_5, \Sigma_4, A_4, D_4\}$, the 
inverse image $p_*^{-1}\cV^{SO(3)}_{H}$ is a Zariski clopen subset
dominated by $\Ht\times_{C_2}Z$ so that
$$ \cV^{U(2)}_{\Ht \times_{C_2}Z}=p_*^{-1}\cV^{SO(3)}_{H}=(\cK_{p^{-1}H}^{U(2)})^{\hash}, $$
where, for each $s\geq 1$ the set $\cK^{U(2)}_{p^{-1}H}$ has $b(H)$ conjugacy classes
of subgroups meeting $Z$ in a subgroup of order $s$.
\end{cor}
For the four finite totally isolated subgroups $H$ of $SO(3)$, 
there is an off the shelf model $\cA (\Ht\times_{C_2} Z\st\full)$ for 
full subgroups of the toral group $\Ht\times_{C_2}Z$ \cite{gq1}. 
This can then be adapted to the block $\cV^{U(2)}_{H}$ of $U(2)$ by 
taking account of fusion. As described in Section \ref{sec:dimonemodels}, 
off-the-shelf models for a space of the form $\cK^\hash$ uses additional structure coming 
from Weyl groups.

The group theory is well known. 
\begin{lemma}
  \label{lem:weyleasy}
  For each of the 5 totally isolated subgroups $H$ in $ SO(3)$
  quotient by the centre $Z$ gives 
  an isomorphism 
  $$  W _{U(2)}(\Ht \times_{C_2}Z)\cong  W _{SO(3)}(H)$$
 
The Weyl groups are $W_{SO(3)}(SO(3) )\cong 1$, 
$W_{SO(3)}({A}_5)\cong 1$, $W_{SO(3) }({\Sigma}_4)\cong 1$, 
$W_{SO(3)}({A}_4)\cong C_2$ and  $W_{SO(3)}({D}_4)\cong \Sigma_3$. 
  \end{lemma}

Since the conjugacy classes of $H$ are totally isolated we have the 
model 
$$\cA (SO(3)|(H))=\Q [W_{SO(3)}(H)]\mbox{-modules} $$
for $SO(3)$-spectra with geometric isotropy $H$. 

Inflating to $U(2)$, since $H$ is finite the group 
$p^{-1}(H)=\Ht \times_{C_2}Z$ is 1-dimensional, so we have the off-the-shelf 
model of the form $\cA (\cK , R, W)$ \cite{gq1}.

\begin{lemma}
\label{lem:fusion}
 If two subgroups cotoral in $p^{-1}H$ are conjugate in $U(2)$ they 
 are conjugate by an element of the normalizer $N_{U(2)}(p^{-1}(H)$. 
 Since the centre is irrelevant to conjugation, the spaces of subgroups are as follows 
$$\xymatrix{
  \cK^{U(2)}_{p^{-1}H}\ar@{=}[r]\dto^{\hash}&\cK^{p^{-1}H}_{p^{-1}H}/W_{SO(3)}(H)\dto^{\hash}&
  \cK^{p^{-1}H}_{p^{-1}H}\ar@{->>}[l] \dto^{\hash}\\
    \cV^{U(2)}_{p^{-1}H}\ar@{=}[r]& \cV^{p^{-1}H}_{p^{-1}H}/W_{SO(3)}(H)&
  \cV^{p^{-1}H}_{p^{-1}H}\ar@{->>}[l]
}$$
The first row records the height 0 conjugacy classes (in $U(2)$ and 
$p^{-1}(H)$) and the second row the compactification. 
\end{lemma}
\subsection{The easiest model}
The very easiest case is when $H=SO(3)$.
The space $\cV^{U(2)}_{SO(3)}$ is  1-dimensional with a 
single generic point, we could argue that the methods of \cite{gq1}
apply (despite the fact that the 
height 0 subgroups are not finite). The set of height 0 subgroups is 
$$\cK=\{ SU(2)\times_2 C_{2s}\st s\geq 1\}.$$
The compactifying point is $U(2)$ itself. The Weyl groups of subgroups 
in $\cK$ are all copies of the circle, and the Weyl group of $U(2)$ is 
trivial; all of these are connected so the component structure is 
trivial. This is a classic Type 1 example.

However it is much  simpler to point out 
that $\cV^{U(2)}_{SO(3)}$ consists of the subgroups containing the normal
subgroup $SU(2)$.
We may then use the traditional passage to geometric fixed points. 
\begin{lemma}
 Passage to $SU(2)$-fixed points gives an equivalence 
$$\mbox{$U(2)$-spectra$\lr{\cV^{U(2)}_{SO(3)}}$}\simeq 
  \mbox{$T$-spectra}$$
  corresponding to the identification 
  $$\cA (U(2), \Q[c])\simeq \cA (T) . $$
\end{lemma}

The required model then follows from \cite{tnqcore}.
\begin{cor}
  There is a Quillen equivalence
  $$\mbox{$U(2)$-spectra $\lr{\cV^{U(2)}_{SO(3)}}$}\simeq DG-\cA
  (\cV^{U(2)}_{SO(3)}, \Q[c]).$$
  \end{cor}

\subsection{Five easy models}
We have already dealt with $H=SO(3)$, but since the statements in this 
section apply to it we will include that in the statements, just giving the 
proofs for the remaining four toral cases. 

To prepare for the statement about models we describe the analogous
situation for a finite group $G$. If we
have $K\subseteq H\subseteq G$ and if all $G$-conjugates of $K$ lies
inside $H$ then the $G$-conjugacy class of $K$ breaks up into a number
of different $H$-conjugacy classes
$$(K)_G=\coprod_i (K_i)_H$$
where $K_i=K^{g_i}$. Now 
Given  a representation $V$ of $W_G(K)$, we may view it  as a
$G$-equivariant bundle $\cV$ over $(K)_G$ fixed by $K$. Over $K^g$ we have the
representation $(g^{-1})^*V$ of $K^g$ fixed by $K^g$. Thus $V$ restricts to
$((g_i^{-1})^*V)_i$. It is then natural to factorize restriction from
$G$ to $H$ as follows (where the superscript indicates that in the
fibre over $K$, the group $K$ fixes the bundle pointwise (the bundle
is `Weyl', in the sense of Barnes)).
\newcommand{\bundle}{\mbox{-bundle}}
$$\xymatrix{
 G\bundle^f/(K)_G\rto \modules \ar@{=}[d]& H\bundle^f /(K)_G\rto^{\pi_1}\ar@{=}[d]
 &H\bundle^f/(K)_H \ar@{=}[d]\\
\Q W_G(K) \modules \rto &\prod_i \Q W_H(K_i)\modules \rto^{\pi_1}&\Q W_H(K) \modules
}$$

With this preparation, the statement in $U(2)$ should seem natural. 
\begin{prop}
For each of the totally isolated subgroups $H$ of $SO(3)$, we have the
following models. 

(i) An abelian  model for $p^{-1}H$-spectra with geometric isotropy the
full subgroups  is the Type 1 model with trivial component structure
$$\cA (p^{-1}H\st \cV^{p^{-1}H}_{p^{-1}H})=\cA 
(\cK^{p^{-1}H}_{p^{-1}H}, \Q[c]). $$

(ii) An abelian  model for $U(2)$-spectra with geometric isotropy in 
$\cV^{U(2)}_{p^{-1}H}$ is the Type 1 model 
$$\cA (U(2)\st \cV^{U(2)}_{p^{-1}H})=\cA 
(\cK^{U(2)}_{p^{-1}H}, \Q[c])[W_G(H)]. $$
The group $W_G(H)$ acts trivially on the rings. 

(iii) These models are related as follows under restriction from
$U(2)$ to $p^{-1}(H)$
$$\xymatrix{
  \cA (U(2) \st \cV^{U(2)}_{p^{-1}H})\ar@{=}[d]\rto
  &\cA (p^{-1}H \st \cV^{p^{-1}H}_{p^{-1}H}) \ar@{=}[d]\\
  \cA (\cK^{U(2)}_{p^{-1}H}, \Q[c], W_{SO(3)}(H))\rto\dto^\simeq
  & \cA (\cK^{p^{-1}H}_{p^{-1}H}, \Q[c], 1)\dto^\simeq\\
    \cA (\cK^{U(2)}_{p^{-1}H})[W_{SO(3)}(H)]\rto
& \cA (\cK^{p^{-1}H}_{p^{-1}H})
}$$
where the horizontal forgets the action of $W_G(H)$ and is pullback
along the quotient map $\cK^{U(2)}_{p^{-1}H} \lla
\cK^{p^{-1}H}_{p^{-1}H} $
(so $\Q[c]$-module over the orbit in the $U(2)$-model is repeated over
each of the subgroups in the orbit). 
\end{prop}


\begin{proof}
  Parts (i) and (ii) are special cases of the results of \cite{gq1}.

Part (iii) follows from Lemma \ref{lem:fusion}. 
  \end{proof}

\section{Subgroups of the maximal torus}
\label{sec:subT}
For the remaining two components, those dominated by the maximal torus
$\T$ or its normalizer $\N$,
it will be useful to have a clear description of the space $\sub(\T)$ with the action of
the Weyl group $W$ of order 2. This is a summary of constructions from
\cite{t2wqalg} in our special case, and we refer readers there for
more details. 

The obvious thing is that two circle subgroups are invariant under $W$:  the subgroup
$Z$ of  scalar matrices and the subgroup $\Tt$ of matrices $\diag (z,
z^{-1})$. Thus $Z\cap \Tt$ is of order 2, and $\T=\Tt\times_2 Z$
(where $\times_2$ indicates the quotient by the central subgroup of
order 2).

\subsection{Pontrjagin duality}
The following summary of results from \cite{t2wqalg} should be
self-contained, but it is discussed at greater length there, especially
in Section 10.
 
To understand the space of subgroups, we  apply Pontrjyagin
duality. For the group $U(2)$ we see there is an isomorphism
$\T^*\cong \Z W$ of $\Z W$-modules.  
If $S\subseteq \T$ we have the corresponding subgroup
$S^{\dagger}:=(\T/S)^*\subseteq \T^*=\Z W$ (the dependence on $\T$ is
not displayed in the notation). The dagger construction is
order reversing and cotoral inclusions of subgroups $S$ correspond to
cofree\footnote{The subgroup $B$ of the abelian group $A$ is {\em cofree}
  if $A/B$ is free} inclusions of lattices $S^\dagger$.

The only 2-dimensional subgroup is $S=\T$,  and we have $\T^{\dagger}=0$.

The one dimensional subgroups $S$ correspond to cyclic subgroups of 
$\Z W$. A cyclic subgroup is invariant  if the generators are multiples of $1-w$ or 
of $1+w$. By virtue of the dagger construction $1-w$ (the $-1$
eigenspace) corresponds to the centre $Z$
whilst $1+w$ (the $+1$ eigenspace) corresponds to $\Tt$.

The finite subgroups $S$ correspond to rank 2 lattices 
$\Lambda \subseteq \Z W$. The $W$-invariant lattices come in two
types,  each of them parametrised by a pair of positive integers. Those of 
Type 1 are $\Lambda_1(m,n)=\langle m(1+w), n(1-w)\rangle$ where
$m,n\geq 1$. Those of Type  2 are  $\Lambda_2(m,n)$ when $m+n$ and $m-n$ are
both even. The lattice $\Lambda_2(m,n)$ 
 has $\Lambda_1(m,n)$ as a submodule of 
index 2, with the second coset containing $((m+n)/2+(m-n)/2\cdot w)$.

The other lattices are not $W$-invariant. 

\subsection{Depiction of the subgroups}
The conventions are determined by earlier choices. For these and
further details see \cite[Section 10]{t2wqalg}. 

We will display the invariant subgroups in a square parametrized by
$(m,n)$ where $1\leq m,n \leq \infty$, and this is
imagined as $(m,n)$ being in the third quadrant $[-1,0]\times [-1,0]$
with  $(m,n)$ at $(-1/n, -1/m)$. This is arranged so that cotoral
inclusions go up the page. 

Along the top edge,   $m=\infty$ with $(\infty, n)$ corresponding to the 
subgroup $C_{2n} \times_2  Z$ (this is generated by $Z$ together with
the elements $\diag (\zeta , 1)$ with $\zeta^n=1$, and is isomorphic 
to the product of $C_n$ and a circle
). Along the right hand vertical, 
$n=\infty$ with $(m,\infty)$ corresponding to the 
subgroup $\Tt \times_2 C_{2m}$ (this is generated by $\Tt$ together with
the elements $\diag(\zeta , 1)$ with $\zeta^m=1$, and is isomorphic 
to the product of $C_m$ and a circle). If $m, n$ are finite, the subgroups 
corresponding to $\Lambda_1(m,n)$ occurs, and if $m+n, m-n$ are even
there is a second subgroup at $(m,n)$ corresponding to 
 $\Lambda_2(m,n)$. 

The non-invariant subgroups are depicted separately (if at all!). 

\subsection{Occurences of the subgroups}
The list of subgroups of $\T$ is relevant in two ways. The first is
obvious and the second requires explanation.

\subsubsection{Direct occurrence}
We have just described $\cV^\T_\T=\sub(\T)$, and if we work with a
larger ambient group $H$ (i.e., $\T \subseteq H\subseteq G$) then the map
$\cV^{\T}_{\T}\lra \cV_{\T}^H$ describing fusion under $H$-conjugacy is surjective.
For example we have $\cV^{\N}_{\T}=\sub(\T)/W$, and we
will show in Lemma \ref{lem:nofusion} that there is no further
fusion  in passage to $U(2)$ 
so that the natural inclusion gives an identification $\cV^{U(2)}_{\T}=\sub(\T)/W$. 

\subsubsection{Parametrising full subgroups}
The subgroups of $\T$ are also relevant to spaces of subgroups
dominated by $\N$ (both amongst subgroups in $\N$ and amongst
subgroups in $U(2)$).

When the ambient group is $\N$, the subgroups dominated by $\N$ are
the full subgroups. The classification up to conjugacy of full subgroups of toral
groups is described in \cite{t2wqalg}. In effect each full subgroup
$H$ is associated to the $W$-invariant subgroup  $S=H\cap \T$, and we
recover $H=\langle S, \sigma\rangle$ for some $\sigma\in U(2)$ over
the generator of $W$.  Equivalently $S$ is specified by the $W$-invariant lattice $S^{\dagger}\subseteq
\T^*=\Z W$ so that $\cV^{\N}_{\N}$ lies
over the picture of $W$-invariant subgroups of $\T$ described in the two previous
subsections. In fact there are finitely many
conjugacy classes $(H)_{\N} \in \fX_{\N}$ over each lattice $S^{\dagger}\in
\mbox{$W$-$\sub (\Z W)$}\cong \mbox{$W$-$\sub(\T)$}$, and the multiplicity of conjugacy
classes above each subgroup $S$ is given by a cohomology calculation. In the
present case the multiplicity is 2 over $(m,\infty)$ and over
the lattices $\Lambda_1(m,n)$ (one split extension and one non-split) 
and the multiplicity is 1 otherwise. These groups are denoted
$H_1^s(m,n), H_1^{ns}(m,n)$ and $H_2(m,n)$ respectively; when
discussing all cases together, we write $H^{\lambda}(m,n)$ with $\lambda \in \{
(1,s), (1,ns), 2\}$. 

When the ambient group is $U(2)$, we warn that from its definition,
the 1-dimensional groups in $\cV^\utwo_\N$ only include
subgroups mapping to dihedral groups $D_{2n}$ for $n\geq 3$. Writing
$\cV^{\N}_{\N, \geq 3}$ for the
subgroups over $(m,n)$ with $n\geq 3$,  fusion gives a map to
$\cV^\utwo_\N$, and we show in Lemma \ref{lem:nofusion} that it induces a
bijection $\cV^\N_{\N, \geq 3}=\cV^\utwo_\N$.

\begin{lemma}
  (i)  The subgroups of $\N$ mapping to $D_{2n}$ in $SO(3)$ are
  precisely those full subgroups 
corresponding to lattices $\Lambda_i(m,n)$ for $1\leq m \leq \infty$
and $i=1, 2$.

(ii) For $n=1$, these are exactly the  abelian full subgroups of $\N$. 
  \end{lemma}

  \begin{proof}
(i)     The subgroups $C_n\subseteq SO(2)\subseteq O(2)$ in $SO(3)$ have
        inverse images $C_{2n}\subseteq \Tt \subseteq \N$ in $SU(2)$.
Thus the maximal subgroups $H\subseteq \N$ mapping to $D_{2n}$
with $S=H\cap \T$ are precisely those with $S=C_{2n}\times_2 Z$, with $S^\dagger$
generated by $n(1-w)$. Subgroups of these are the ones with lattices
containing $n(1-w)$.

(ii)     The subgroups $H$ corresponding to $S=H\cap \T$ are generated by $S$
and an element $w$, which acts on $S$ according to its image in
$W$. The group is therefore abelian if and only if $W$ acts trivially
on $S$, or equivalently, if it acts trivially on $(\Z W)/\Lambda$, so
that $1+\Lambda=w+\Lambda$, which is to say $1-w\in \Lambda$.
    \end{proof}


\section{Normalizers and fusion}
\label{sec:normfus}
We need to relate $\N$-conjugacy  to $U(2)$-conjugacy for subgroups of
$\N$. The starting point is as follows. 

\begin{lemma}
\label{lem:nofusion}
 If $S,S'\subseteq  \T$ are conjugate by an element of $U(2)$ outside 
 $\N$ then both groups are central and $S=S'$. In particular, if $S$
 is not central $N_{U(2)}(S)=N_{\N}(S)$.
\end{lemma}

\begin{proof}
If $g\not \in \N$ then $\T\cap \T^g=Z$. Hence $S'=S^g\subseteq
Z$ and hence $S=S'$. 
\end{proof}


In $U(2)$ any abelian group is conjugate to a subgroup of $\T$, so any
of the abelian groups $H^\lambda (m,n)$  are conjugate to subgroups of
$\T$. This means that any of the subgroups $H^\lambda(m,n)$ with $n\leq 2$
occur in $\cV^\utwo_\T$. 

\begin{lemma}
  \label{lem:nofusionfull}
 If $H$ is a non-abelian full subgroups of $\N$ and $g\not\in \N$ then
 $H^g$ does not lie in $\N$. In particular $N_\utwo (H)=N_\N(H)$. 
\end{lemma}

\begin{proof}
First note that $\N$ consists of diagonal matrices $\T$ together with
matrices
$A=\left( \begin{array}{cc} 0&\lambda\\ \mu&0\end{array}\right)$, so
that $A^2\in Z$. 
  
Suppose $S=H\cap \T$ so that $H=\langle S , w\rangle=S\amalg
wS$ for some $w$ mapping to a generator of the Weyl group. 

Since $H^g=H'\subseteq \N$ we have $S\subseteq S^g \cup (wS)^g$.
If $g\not \in \N$ then $\T \cap \T^g\subseteq Z$. By the preamble
$(wS)^g\cap \T$ has square in the centre. Altogether $S$ consists of
elements whose square lies in the centre. This means $H$ is
abelian and hence not under consideration. 
\end{proof}

\section{The component dominated by the maximal torus}
\label{sec:domT}
In this section we describe the component $\cV^{U(2)}_{\T}$ dominated
by the maximal torus. The model for the toral part of $G$-spectra
 has been described in detail in \cite{AGtoral} for arbitrary $G$, and
 shown to be a model in \cite{gtoralq}.

 It is also shown that for a general group $G$ 
that $\cA (\N|\full)\simeq \cA (\T)[W]$, and that
 $\cA(G|\cV^G_\T)$ is a retract of $\cA (\N |\full)$ in the sense that
 there is an adjunction
 $$\adjunction{\theta_*}{\cA (G|\cV^G_{\T})}{\cA (\N
   |\cV^\N_{\T})}{\Psi}$$
 with unit an isomorphism \cite[Corollary 7.11]{AGtoral}, and
 similarly for $G$-spectra.  The difference arises because Weyl groups
 are different, so that both the sheaf of rings and the component
 group can change.

Turning to $G=U(2)$, the space of subgroups of $\T$ along with the $W$ action has been
described in Section \ref{sec:subT}. 

The algebraic model uses cotoral flags of subgroups, which have form
$$(S_0), (S_0>S_1), (S_0>S_1>S_2). $$
Equivalently, we may consider the Pontrjagin dual objects, which are
cofree flags
$$(S_0^{\dagger}) , (S_0^{\dagger} <S_1^{\dagger}),
(S_0^{\dagger}<S_1^{\dagger}<S_2^{\dagger})$$
of subgroups of $\Lambda=\Z W$. The $W$-fixed flags are those with all
terms $W$-invariant.

Attached to a flag $F$ we have the group $W_G(F)=(\bigcap_i
N_GS_i)/S_{max}$, which in turn gives a sheaf of rings with
$\mcR(F)=H^*(BW^e_G(F))$ and a component structure $\cW_F:=W_G^d(F)$.

\begin{lemma}
  The normalizers of subgroups of $\T$ are as follows
  \begin{itemize}
      \item $N_{U(2)}(B)=U(2)$, $N_{\N}(B)=\N$  if $B\subseteq Z$
      \item For $2\leq n< \infty$ we have 
        $N_{U(2)}(H_1^{s/ns}(m,n))=N_{\N}(H_1^{s/ns}(m,n)))=H(\infty , 
        2n)$
         \item For $2\leq n< \infty$ and $m$ with $m+n, m-n$ even we have 
        $N_{U(2)}(H_2(m,n))=N_{\N}(H_2(m,n)))=H(\infty , n)$
        \item 
        $N_{U(2)}(H^{s/ns}(m,\infty))=N_{\N}(H^{s/ns}(m , \infty)))=\N$
   \end{itemize}

 \end{lemma}

 \begin{proof}
  By Lemma \ref{lem:nofusion}, we have $N_{U(2)}(K)=N_{\N}(K)$ unless 
  $K\subseteq Z$ when $N_{U(2)}=U(2)$.

  For the rest we point to \cite[Section 10]{t2wqalg}.
\end{proof}

 We see from this that the difference between toral $U(2)$-spectra and
 toral $\N$-spectra is entirely attributable to central subgroups
 where we have
  $$\mcR_{\N }(Z[n])=H^*(B(\N/Z[n])_e)=\Q [c,c'] \mbox{ and }
 \mcR_{U(2)}(Z[n])=H^*(B(U(2)/Z[n])_e)=\Q [c,d']$$
 where $c,c'$ are of codegree 2 and $d'=(c')^2$. 
 Correspondingly 
 $$\cW_{\N, Z[n]}=C_2 \mbox{ and } \cW_{U(2), Z[n]}=1. $$

Similarly
 $$\mcR_{\N }(Z)=H^*(B(\N/Z)_e)=\Q [c'] \mbox{ and }
 \mcR_{U(2)}(Z)=H^*(B(U(2)/Z)_e)=\Q [d']$$
 and 
 $$\cW_{\N, Z}=C_2 \mbox{ and } \cW_{U(2), Z}=1. $$

\section{The component dominated by the maximal torus normalizer}
\label{sec:domN}
The normalizer of the maximal torus is the toral group 
$$\N = N_{U(2)}(\T)=\T\sdr C_2$$
where the group of order two exchanges the two circle factors. 

In Subsections \ref{subsec:NNmodel1} and
\ref{subsec:NNmodel3}, 
we describe the model for $\N$
spectra with full isotropy given in \cite[Section 2]{t2wqmixed}.

There is an idempotent in the Burnside ring of $U(2)$ with support
consisting of subgroups $H^\lambda(m,n)$ with $n\leq 2$, so we may omit these
subgroups corresponding from the geometric isotropy in spectra and
models in $\N$  without any essential change. 

Thanks to the group theoretic behaviour this also gives a
model for $U(2)$-spectra.

\subsection{Spectra dominated by $\N$ are the same over $U(2)$ as
  over $\N$}
\label{subsec:NNmodel1}
The simplest possible relationship holds between the two categories. 
\begin{thm}
 Restriction from $U(2)$-spectra to $\N$-spectra is an equivalence of
 categories
 $$\mbox{$U(2)$-spectra $\lr{\cV^\utwo_\N}$}\simeq
 \mbox{$\N$-spectra $\lr{\cV^\N_{\N,\geq 3}}$}$$
 of rational spectra with geometric isotropy dominated by $\N$. 
  \end{thm}
  \begin{proof}
We must show restriction is full, faithful and essentially
surjective.

\begin{lemma}
  \label{lem:u2nff}
 Restriction from $U(2)$-spectra to $\N$-spectra is full and faithful 
 on spectra with geometric isotropy dominated by $\N$: restriction 
 $$[X,Y]^G\lra [X,Y]^\N$$
 is an isomorphism whenever $X$ and $Y$ have geometric isotropy in 
 $\cV^{U(2)}_{\N}$.  
\end{lemma}

\begin{proof}
Since 
$ [X,Y]^\N=[G_+\sm_\N X, Y]^G $, it suffices to observe
$G/\N_+ \lra S^0 $ is an equivalence in $K$-geometric fixed points 
whenever $K$ is a nonabelian full subgroup of $\N$. Since
$(G/\N)^K=\{ g\N \st K^g\subseteq \N\}$, Lemma \ref{lem:nofusionfull}
shows it has a single point for each $K$.
\end{proof}

\begin{lemma}
  Suppose $H$ is a subgroup of $G$ and $\mcH$ is a collection of subgroups
  of $H$ closed under $G$-conjugacy. A necessary and sufficient
  condition for the restriction
  $\res^G_{H}: \mbox{$G$-spectra$\lr{\mcH}$}\lra  \mbox{$H$-spectra$\lr{\mcH}$}$
  to be essentially surjective up to retracts is that for all $K\in \mcH$ the map
  $W_H(K)\lra  W_{G}(K)$ is injective on component groups. 
\end{lemma}

  \begin{proof}
    From \cite{spcgq}, we know that the category of $G$-spectra with
    geometric isotropy the singleton $K$ is generated by the
    $G$-spectrum $A_G\lr{K}:=G_+\sm_N E_N\lr{K}$ where $N=N_G(K)$
    corresponding to the torsion
    $W_G(K)$-module $H_*(BW^e_G(K))[W_G^d(K)]$, and the category of
    $G$-spectra over any collection $\mcH$ of subgroups is generated
    by the $A_G\lr{K}$ for $K\in \mcH$. Thus the category of
    $G$-spectra dominated by $\N$ is generated by $A_G\lr{K}$ for 
    $K\in \mcH$ and the category of $H$-spectra with geometric
    isotropy in $\mcH$  is generated by $A_{H}\lr{K}$ for     $K\in \mcH$.

    By hypothesis, the restriction of $A_G\lr{K}$ to $H$ is a finite
    wedge of copies of $A_H\lr{K}$.
\end{proof}

\begin{cor}
 Restriction from $U(2)$-spectra to $\N$-spectra is essentially surjective 
 on spectra with geometric isotropy dominated by $\N$. 
\end{cor}

\begin{proof}
  We apply the lemma with $G=U(2), H=\N$, and then Lemma \ref{lem:u2nff}
  to ensure idempotents are realised. 
    \end{proof}

\end{proof}

\subsection{The model for $\N$-spectra with full geometric isotropy}
 \label{subsec:NNmodel3}

The 
{\em standard model} gives horizontal and vertical
directions equal status, and collects subgroups acording to their
Thomason height (which is their dimension in this case).
Objects of the standard model are diagrams $N$  of $\cOcFt$-modules, where
$$
\cOcFt=\left(
\begin{gathered}
\xymatrix{
&\cEi\cIi \cOcF&\\
\cEi \cOcF\urto &&\cIi \cOcF\ulto \\
&\cOcF\urto \ulto&}
\end{gathered}\right)$$
Thus $N$ is the diagram
$$\xymatrix{
&N(\up G)&\\
N(\up Z)\urto &&N(\up \Tt)\ulto \\
&N(\up 1)\urto \ulto&}$$
which is quasicoherent and extended
\begin{enumerate}
\item {\em quasicoherence} is the requirement that 
$$N(\up Z)=\cEi 
  N(\up 1), N(\up \Tt)=\cIi 
  N(\up 1), N(\up G)=\cEi\cIi   N(\up 1)$$
\item {\em extendedness} is the requirement that 
$$
N(\up Z)=\cEi \cOcF\tensor_{\cOcFZ} \phi^ZN, 
N(\up \Tt)=\cIi \cOcF\tensor_{\cOcFTt} \phi^{\Tt} N,
N(\up G)=\cIi\cEi \cOcF\tensor \phi^GN, $$
where $\phi^ZN$ is an $\cOcFZ$-module,
$\phi^{\Tt} N$ is an $\cOcFTt$-module,
$\phi^GN$ is an $\Q$-module, 
\end{enumerate}

For the mixed component, the category of spectra and the algebraic
model are minor adaptions of those for the torus normalizer
itself. Accordingly, the arguments of \cite{t2wqmixed} apply to show this
gives a model. 


\bibliographystyle{plain}
\bibliography{../../jpcgbib}
\end{document}